\documentclass[preprint,12pt]{article}

\usepackage{amssymb, amsmath, bbm, mathrsfs, color,url}
\usepackage{fullpage, amsthm}

\theoremstyle{plain}
\newtheorem{thm}{Theorem}[section]
\newtheorem{cor}[thm]{Corollary}
\newtheorem{prop}[thm]{Proposition}

\theoremstyle{definition}
\newtheorem{defn}[thm]{Definition}
\newtheorem{lem}[thm]{Lemma}
\newtheorem{ex}[thm]{Example}

\newcommand{\bR}{\ensuremath{\mathbb{R}}}
\newcommand{\bC}{\ensuremath{\mathbb{C}}}

\newcommand{\bF}{\ensuremath{\mathbb{F}}}

\newcommand{\bN}{\ensuremath{\mathbb{N}}}
\newcommand{\cW}{\ensuremath{\mathcal{W}}}
\newcommand{\cV}{\ensuremath{\mathcal{V}}}

\newcommand{\cZ}{\ensuremath{\mathcal{Z}}}

\newcommand{\beq}{\begin{equation}}
\newcommand{\eeq}{\end{equation}}

\newcommand{\ip}[2]{\ensuremath{\langle {#1},{#2} \rangle}}

\newcommand{\lspan}{\operatorname{span}}

\newcommand{\col}{\operatorname{col}}
\newcommand{\tr}{\operatorname{tr}}

\begin{document}

\title{Constructing Subspace Packings from Other Packings}
\author{Emily J.~King\\Colorado State University\\\url{emily.king@colostate.edu}}

\maketitle

\begin{abstract} The desirable properties when constructing collections of subspaces often include the algebraic constraint that the projections onto the subspaces yield a resolution of the identity like the projections onto lines spanned by vectors of an orthonormal basis (the so-called tightness condition) and the geometric constraint that the subspaces form an optimal packing of the Grassmannian, again like the one-dimensional subspaces spanned by vectors in an orthonormal basis. In this article a generalization of related constructions which use known packings to build new configurations and which appear in numerous forms in the literature is given, as well as the characterization of a long list of desirable algebraic and geometric properties which the construction preserves.  Another construction based on subspace complementation is similarly analyzed. While many papers on subspace packings focus only on so-called equiisoclinic or equichordal arrangements, attention is also given to other configurations like those which saturate the orthoplex bound and thus are optimal but lie outside of the parameter regime where equiisoclinic and equichordal packings can occur.
\end{abstract}

\section{Introduction}
\subsection{Motivation}

The goal is to find optimal configurations of subspaces which are of interest mathematically but also in applications such as coding theory (see, e.g., \cite{Cre08,PWTH16,XZG05,KaPa03,GrassFus}), quantum information theory (see, e.g., \cite{FHS17,AFZ15,ShSl98,GoRo09}), and more.  The usefulness of such configurations often comes from whether the projections onto the subspaces (approximately) yield a resolution of the identity and whether the angles between the subspaces are as large as possible.  Configurations which satisfy the former condition are called (tight) fusion frames and the latter condition are called Grassmannian packings.

In general, Grassmannian fusion frames, which are fusion frames which correspond to optimal packings of Grassmannian spaces with respect to the chordal distance, are optimally robust against noise and erasures. Under certain models of noise and erasures, a type of Grassmannian fusion frame called equichordal is shown to be optimal \cite{GrassFus}. Under other models, the subclass of Grassmannian fusion frames which are called equiisoclinic have been proven to be the best \cite{Bod07}, while for certain coding theory regimes such packings are not optimal \cite{PWTH16}.  Equiisoclinic tight fusion frames are also proven to be deterministically optimal in block sparse recovery \cite{SAH14,EKB10}. 

We complete this section by introducing Grassmannian fusion frames and basic notation (Section~\ref{sec:frame}). In Section~\ref{sec:recycle}, we give generalizations, classifications, and examples of constructions from the literature of various Grassmannian packings that can be constructed from other such packings.  In particular, a construction of Grassmannian packings and fusion frames from ones known to exist which generalizes different constructions based on Kronecker products in \cite{LemSei73,BCPST,Cre08,SAH14,CFMWZ,CTX15,EKB10,MKH20} is presented in Theorem~\ref{thm:gentens} and Corollary~\ref{cor:gentens}, coupled with proofs of which desirable algebraic and geometric properties are inherited from the original packings.  Then subspace complementation, a construction method found in, e.g.\ \cite{BCPST}, is characterized including a focus on optimal Grassmannian packings which are not equichordal (Proposition~\ref{prop:orth}).

Throughout the paper, $\bF$ will always either denote $\bR$ or $\bC$.  Further for $m, n \in \bN$, we define $[n] := \{1, 2, \hdots, n\}$, $M(\bF,m,n)$ to be the set of $m\times n$ matrices with entries in $\bF$, and $I_n$ to be the $n\times n$ identity matrix. Finally, for $A \in M(\bF,m,n)$, we write $\col(A)$ for the column span of $A$.  
\subsection{Fusion Frames and Grassmannian Packings}\label{sec:frame}
Our objects of interest are collections of subspaces, which may be viewed as points in a Grassmannian.
\begin{defn}\label{defn:chor}
For $1 \leq m \leq k-1$, set $Gr(\bF,k,m)$ to be the collection of $m$ dimensional subspaces of $\mathbb{F}^k$.  $Gr(\bF,k,m)$ is called a \emph{Grassmannian}.  $Gr(\bF,k,m)$ is endowed with a metric space structure induced by the \emph{chordal distance} (see, e.g., \cite{GrassPack})
\begin{equation}\label{eqn:chord}
d_c(\mathcal{W}_i,\mathcal{W}_j) = [m - \tr(P_i P_j)]^{1/2},
\end{equation}
for $\mathcal{W}_i, \mathcal{W}_j \in Gr(\bF,k,m)$, where $P_i$ is the orthogonal projection onto $\mathcal{W}_i$.  
\end{defn}
We do not consider the trivial cases $m = 0$ and $m=k$ as they are not interesting. We also note that there are many other metrics that one can endow the Grassmannian with, like the spectral distance, the geodesic distance, and the Fubini-Study distance (see, e.g., \cite{GrassPack,DHST08}). 
Let us consider a set of vectors $E=\{e_i\}_{i=1}^k$ for $\bF^k$ and define for each $i \in [k]$ $\cW_i = \lspan \{e_i\} \in Gr(\bF,k,1)$.  Then $E$ is an orthonormal basis precisely when  $P_i = e_i e_i^\ast$ is the orthogonal projection onto $\cW_i$ for each $i$, $I_k = \sum_{i=1}^k P_i$, and
\beq\label{eqn:grasspackonb}
\min_{i,j \in [k], i\neq j} d^2_c(\mathcal{W}_i,\mathcal{W}_j) = \max_{\{\mathcal{V}_i\}_{i=1}^k \subset Gr(\bF,k,1)} \left( \min_{i,j \in [k], i\neq j} d^2_c(\mathcal{V}_i,\mathcal{V}_j) \right).
\eeq
We now generalize these traits to systems which may be overcomplete and consist of subspaces of dimension greater than $1$. Collections of possibly redundant systems of subspaces which yield algebraic reconstruction properties akin to orthonormal bases have appeared in the literature for many years with many names, including \emph{stable space splittings of Hilbert spaces}~\cite{Osw94,Osw97}, \emph{systems of bounded quasi-projectors}~\cite{Forn03,Forn04}, \emph{(weighted projective) resolutions of the identity}~\cite{Forn03,Bod07}, \emph{$g$-frames}~\cite{Sun06}, and \emph{frames of subspaces}~\cite{CaK04}. We will use the name \emph{fusion frames} and the related terminology, which were introduced in~\cite{CKL08}. See \cite[Chapter 13]{CaKBook} for a general overview of fusion frames.
\begin{defn}
A finite collection of subspaces $\lbrace \mathcal{W}_i \rbrace_{i=1}^n \subset Gr(\bF,k,m)$ is a \emph{tight fusion frame of $m$-dimensional subspaces with unit weights} for $\mathbb{F}^k$ if there exists an $A > 0$ (called the \emph{fusion frame bound}) satisfying
\begin{equation}\label{eqn:tight4}
x = \frac{1}{A} \sum_{i=1}^n P_i x, \quad \textrm{for all $x \in \bF^k$},
\end{equation}
where $P_i$ is the orthogonal projection onto $\mathcal{W}_i$. The map $x \mapsto \sum_{i=1}^n P_i x$ is called the \emph{fusion frame operator}.
\end{defn}
One may loosen this definition by allowing non-equidimensional subspaces, non-equal weights, and the fusion frame operator to only be an approximation of the identity; however, we will not be concerned with these cases in this paper. To avoid being verbose, we shall refer to tight fusion frames of $m$-dimensional subspaces  with unit weights as \emph{tight fusion frames}. Given $\{\cW_i\}_{i=1}^n \subset Gr(\bF,k,m)$, we fix for each $i \in [n]$ an orthonormal basis $\{e_j^i\}_{j=1}^m$ for the subspace $\cW_i$ and denote by $L_i$ the matrix $(e_1^i e_2^i \hdots e_m^i) \in M(\bF,k,m)$. We further define
\[
L  = \left( \begin{array}{cccc} L_1 & L_2 & \hdots & L_n\end{array}\right).
\]
Then it is clear that~\eqref{eqn:tight4} holds, i.e., that $\lbrace \mathcal{W}_i \rbrace_{i=1}^n$ is a tight fusion frame of $d$-dimensional subspaces with unit weights for $\bF^k$ with fusion frame bound $A$ precisely with the rows of $L$ are orthogonal with norm $\sqrt{A}$.  
 
We will consider three definitions of ``equal'' geometric spread between subspaces.  See \cite{LemSei73,BEt07,BCPST,BjGo73,Cre08}, in particular [Theorem 2.3] of \cite{LemSei73}, [Theorem 1] of \cite{BjGo73}, and [p. 4] of \cite{Cre08}.
\begin{defn}\label{defn:diffeq} 
Let $\{ \cW_i\}_{i=1}^n \subset Gr(\bF,k,m)$ (not necessarily a fusion frame) with corresponding orthonormal bases as the columns of $\{L_i\}_{i=1}^n$. Then we say
\begin{itemize}
\item $\{ \cW_i\}_{i=1}^n$ is \emph{equichordal} when for all $i,j \in [n]$ with $i \neq j$, $\tr(L_i^\ast L_j L_j^\ast L_i)$ is constant;
\item $\{ \cW_i\}_{i=1}^n$ is \emph{strongly simplicial} when for all $i,j \in [n]$ with $i \neq j$, $L_i^\ast L_j L_j^\ast L_i$ has the same set of eigenvalues; and
\item $\{ \cW_i\}_{i=1}^n$ is \emph{equiisoclinic} when there exists an $\alpha \geq 0$ such that for all $i,j \in [n]$ with $i \neq j$, $L_i^\ast L_j L_j^\ast L_i =\alpha I_m$.
\end{itemize}
If $\rho_1, \rho_2, \hdots, \rho_m$ are the eigenvalues of $L_i^\ast L_j L_j^\ast L_i$, then for all $\ell \in [m]$ the $\theta_\ell \in [0, \pi/2]$ which satisfy $\cos^2(\theta_\ell) = \rho_\ell$ are called the \emph{principal angles} between $\cW_i$ and $\cW_j$. 
\end{defn}
We note that given a set of parameters, the only variable in the definition of the chordal distance~\eqref{eqn:chord} is $\tr(L_i^\ast L_j L_j^\ast L_i)=\tr(P_i P_j)$; thus, being equichordal means that $d^2_c(\mathcal{W}_i,\mathcal{W}_j)$ is constant for $i \neq j$, as one would hope given the name. 

If we fix $k$, $m$, $n$ and $\bF$, then the \emph{Grassmannian packing problem} concerns finding $n$ elements in $Gr(\bF,k,m)$ so that the minimal distance between any two subspaces is as large as possible, just like in~\eqref{eqn:grasspackonb}.  An algorithm to approximate solutions is in~\cite{DHST08}, while~\cite{Sloane} has a (somewhat dated) list of best known packings when $\bF = \bR$. For fixed parameters $k$, $m$, $n$ and $\bF$, the maximizers  $\{ \cW_i\}_{i=1}^n$ of $\min_{i,j \in [n], i\neq j} d^2_c(\mathcal{W}_i,\mathcal{W}_j)$ are called \emph{Grassmannian fusion frames} (which may or may not be tight fusion frames).

Definition~\ref{def:Gerz} and Theorem~\ref{thm:simplex} summarize results concerning the Grassmannian packing problem found in~\cite{Ran55,GrassPack,GrassFus,LS1973,BH15,Hen05,FJMW17} and many more sources.
\begin{defn}\label{def:Gerz}
Define
\[
 \cZ(\bF,k) = \left\{\begin{array}{lcr}k^2 & ; & \bF = \bC \\  \frac{k(k+1)}{2} &; & \bF = \bR \end{array} \right.
\]
\end{defn}
This is known as \emph{Gerzon's bound} and comes from the dimension of the smallest vector subspace of $M(\bF,k,k)$ which contains the symmetric / self-adjoint matrices.

\begin{thm}\label{thm:simplex}
Let $\{ \cW_i\}_{i=1}^n \subset Gr(\bF,k,m)$, then
\beq\label{eqn:simplex}
\min_{i,j \in [n], i\neq j}  d^2_c(\mathcal{W}_i,\mathcal{W}_j)  \leq \frac{m(k-m)n}{k(n-1)}.
\eeq
The bound in~\eqref{eqn:simplex} is saturated if and only if $\{ \cW_i\}_{i=1}^n$ is an \emph{equichordal tight fusion frame}. Further, the bound~\eqref{eqn:simplex} can only be saturated if $n \leq  \cZ(\bF,k)$.  When $\cZ(\bF,k) < n \leq 2(\cZ(\bF,k)-1)$, 
\beq\label{eqn:orthoplex}
\min_{i,j \in [n], i\neq j}  d^2_c(\mathcal{W}_i,\mathcal{W}_j)  \leq \frac{m(k-m)}{k}.
\eeq
\end{thm}
Thus if a tight fusion frame is equichordal, strongly simplicial, or equiisoclinic it is a Grassmannian fusion frame since all of those configurations are equichordal.  The bound in~\eqref{eqn:simplex} is known as the \emph{simplex bound}, and the bound in~\eqref{eqn:orthoplex} is the \emph{orthoplex bound}. If the orthoplex bound is saturated and $\cZ(\bF,k) < n \leq 2(\cZ(\bF,k)-1)$, then we call $\{\cW_i\}_{i=1}^n$ an \emph{orthoplectic Grassmannian packing}. Such a packing need not be a tight fusion frame. We further note that when $n \leq  \cZ(\bF,k)$, it is possible for $\{\cW_i\}_{i=1}^n$ to saturate the orthoplex bound but not be a Grassmannian fusion frame.

There are some special terms for the above-defined concepts when $m=1$ (see, e.g., \cite{JKM19,CaKBook,Waldron18}).  Namely, tight fusion frames of $1$-dimensional subspaces with unit weights are called \emph{finite unit norm tight frames}.  Also, the definitions of equichordal, strongly simplicial, and equiisoclinic coincide and are jointly known as \emph{equiangular}.

Some constructions and characterizations of equiisoclinic packings are in \cite{BEt06,BEt07,BEt14,Hog77,LemSei73,CHRSS99,ShSl98}, of strongly simplicial packings are  in \cite{Cre08}, of equichordal packings are in \cite{Cre08,CHRSS99,GrassFus,BoPa15,GrassFusMe,King19a}, and of orthoplectic Grassmannian packings are in \cite{ShSl98,BH15,GoRo09}.


\section{Constructing New Grassmannian Packings from Old}\label{sec:recycle}

A common method for constructing new optimal Grassmannian packings out of already known ones is to use a Kronecker product or Kronecker-like product.  We make note of the following standard definition and properties. 
\begin{defn}\label{def:kron}
For $A \in M(\bF,r, s)$ and $B \in M(\bF,p,q)$, with the $i$th row and $j$th column of $A$ denoted by $a_{i,j}$, we define the \emph{Kronecker product} $A \otimes B \in M(\bF,rp,sq)$ as
\[
A \otimes B = \left( \begin{array}{ccc} a_{1,1} B & \hdots & a_{1,s} B \\ \vdots & \ddots & \vdots \\ a_{r,1} B & \hdots & a_{r,s} B\end{array}\right).
\]
\end{defn}
A certainly incomplete summary of such constructions making use of the Kronecker product in the literature follows. We will often use the more general term \emph{tensor} instead of Kronecker product to circumvent awkward phrasing.  In \cite{LemSei73}, the authors tensor real equiangular lines (not necessarily a frame) with orthogonal matrices to obtain equiisoclinic packings (not necessarily a fusion frame).  They achieve this indirectly by tensoring the Gram matrix of the equiangular lines with an identity matrix ([proof of Theorem 3.7], [Remark 4.2]).  The same construction was repeated in \cite[Theorem 4]{SAH14} but specifically for the case of equiangular tight frames (although the authors falsely assume that all Grassmannian frames are equiangular tight frames). Similarly, in~\cite{MKH20}, orthonormal basis vectors for equichordal tight fusion frames are tensored with an identity to generate equichordal tight fusion frames. The constructions in~\cite{CTX15} involve tensoring vectors in an equiangular tight frame [Theorem 3] or mutually unbiased bases [Theorem 6] (see also~\cite{EKB10}) with a fixed unitary and then considering the spectral distance -- which encourages equiisoclinic packings -- between the subspaces. In~\cite{Hog77}, the concept of \emph{isoclinic covariant functors} as a method of creating new sets of equiisoclinic subspaces from pre-existing ones is introduced.  Tensoring vectors with an identity matrix would be one example of such a functor. Instead of starting with an equiangular tight frame, the author of \cite{Cre08} tensors the projections of an equichordal Grassmannian fusion frame with an identity matrix to obtain another equichordal Grassmannian fusion frame in [Proposition 12].  Finally, in \cite[Corollary 5]{BCPST}, each column vector in a truncated orthonormal basis is tensored with a unitary matrix representing the different subspaces of a particular type of tight fusion frame to obtain a new tight fusion frame, called the Naimark complement, and in [Theorem 7], a specific application of this starting with a collection of orthonormal bases for the entire space is used to obtain a equiisoclinic Grassmannian fusion frame. Furthermore, \cite{CFMWZ} essentially has the construction from Corollary~\ref{cor:gentens}  but in the specific case of building new tight fusion frames from old ones, without a focus on the geometric properties. We now present a construction which subsumes the ones listed above. The theorem also shows that the construction preserves various desired properties separately.
\begin{thm}\label{thm:gentens}
Let $\{\cW_i\}_{i=1}^n$ be a collection of $m$-dimensional subspaces in $\bF^k$.  For each $i \in [n]$, fix an orthonormal basis $\{e_j^i\}_{j=1}^m$ of $\cW_i$.  Further let $\{U_i\}_{i=1}^n$ be a collection of unitaries in $M(\bF,r,r)$.  For each $i \in [n]$ and $j \in [m]$ define 
\begin{align*}
\overline{e}_j^i &= e_j^i \otimes U_i \in M(\bF,rk,r),   \\
\overline{L}_i &= (\overline{e}_1^i \overline{e}_2^i \hdots \overline{e}_m^i) \in M(\bF,rk,rm), \textrm{and}\\
\overline{\cW}_i &= \col \overline{L}_i \in Gr(\bF,rk,rm).
\end{align*}
Then the following statements hold.
\begin{enumerate}
\item For each $i \in [n]$, the columns of $\overline{L}_i$ are a set of $rm$ orthonormal vectors in $\bF^{rk}$;
\item $\{ \overline{\cW}_i\}_{i=1}^n$ is equichordal (not necessarily a fusion frame) if and only if $\{\cW_i\}_{i=1}^n$ is;
\item $\{ \overline{\cW}_i\}_{i=1}^n$ is strongly simplicial (not necessarily a fusion frame) if and only if $\{\cW_i\}_{i=1}^n$ is; 
\item $\{ \overline{\cW}_i\}_{i=1}^n$ is equiisoclinic (not necessarily a fusion frame) if and only if $\{\cW_i\}_{i=1}^n$ is; 
\item $\{ \overline{\cW}_i\}_{i=1}^n$ saturates the orthoplex bound if and only if $\{\cW_i\}_{i=1}^n$ does, but they cannot both be orthoplectic Grassmannian packings if $r\geq 2$;
and
\item $\{ \overline{\cW}_i\}_{i=1}^n$ is a tight fusion frame if and only if $\{\cW_i\}_{i=1}^n$ is. In this case, they have the same fusion frame bound.
\end{enumerate}
\end{thm}
\begin{proof}
We begin by computing blocks of $\overline{L}_i^\ast \overline{L}_{\tilde{i}}$ for $i, \tilde{i} \in [n]$, making use of basic properties of the Kronecker product.  We note that for $j, \tilde{j} \in [m]$,
\begin{align}
(\overline{e}_j^i)^\ast \overline{e}_{\tilde{j}}^{\tilde{i}} &= (e_j^i \otimes U_i)^\ast (e_{\tilde{j}}^{\tilde{i}} \otimes U_{\tilde{i}}) = \left((e_j^i)^\ast (e_{\tilde{j}}^{\tilde{i}})\right) \otimes\left( (U_i)^\ast U_{\tilde{i}} \right)\nonumber\\
&= \ip{e_{\tilde{j}}^{\tilde{i}}}{e_j^i} \otimes\left( (U_i)^\ast U_{\tilde{i}} \right)= \ip{e_{\tilde{j}}^{\tilde{i}}}{e_j^i} \left( (U_i)^\ast U_{\tilde{i}} \right)\label{eqn:crossbloc}
\end{align}
When $i = i'$, Equation~\ref{eqn:crossbloc} simplifies to
\[
(\overline{e}_j^i)^\ast \overline{e}_{\tilde{j}}^{i} =  \ip{e_{\tilde{j}}^{i}}{e_j^i} I_r = \left\{ \begin{array}{ll}I_r; & \tilde{j} = j\\ 0_{r}; & \tilde{j} \neq j \end{array}\right.,
\]
where $0_r$ is the $r \times r$ zero matrix. In this case, the off-diagonal blocks of the Gram matrix $\overline{L}_i^\ast \overline{L}_i$ are the zero matrix and the diagonal blocks are the identity.  Thus the columns of $\overline{L}_i$ are orthonormal, and Statement 1 is proven.\\

To prove Statements 2 -- 5, we are concerned with the properties of $\overline{L}_i^\ast \overline{L}_{\tilde{i}}$, or more specifically $\overline{L}_i^\ast \overline{L}_{\tilde{i}} (\overline{L}_{\tilde{i}})^\ast \overline{L}_i$, when $i \neq \tilde{i}$. As usual, we define $L_i$ to be the matrix with columns $\{e_j^i\}_{j=1}^m$ (from the original packing).  Then the $(j, \tilde{j})$ entry of $L_i^\ast L_{\tilde{i}}$ is precisely $\ip{e_{\tilde{j}}^{\tilde{i}}}{e_j^i}$.  Thus, it follows from  Equation~\ref{eqn:crossbloc} that
\begin{align}
\overline{L}_i^\ast \overline{L}_{\tilde{i}}& = (L_i^\ast L_{\tilde{i}})\otimes U_i^\ast U_{\tilde{i}}, \textrm{ and further}\nonumber\\
\overline{L}_i^\ast \overline{L}_{\tilde{i}} (\overline{L}_{\tilde{i}})^\ast \overline{L}_i & = \left((L_i^\ast L_{\tilde{i}})\otimes U_i^\ast U_{\tilde{i}}\right)  \left((L_{\tilde{i}}^\ast L_{i})\otimes U_{\tilde{i}}^\ast U_{i}\right)\nonumber \\
& = \left(L_i^\ast L_{\tilde{i}}L_{\tilde{i}}^\ast L_{i}\right) \otimes \left( U_i^\ast U_{\tilde{i}}U_{\tilde{i}}^\ast U_{i} \right)\label{eqn:crossbloc2b} \\
& = \left(L_i^\ast L_{\tilde{i}}L_{\tilde{i}}^\ast L_{i}\right) \otimes I_r. \label{eqn:crossbloc2}
\end{align}

It follows from Definition~\ref{defn:diffeq}, that $\{\overline{\cW}\}_{i=1}^n$ is equichordal if and only if $\tr \left(\overline{L}_i^\ast \overline{L}_{\tilde{i}} (\overline{L}_{\tilde{i}})^\ast \overline{L}_i\right)$ is constant for all $i \neq \tilde{i}$, which happens if and only if $\tr  \left(L_i^\ast L_{\tilde{i}}L_{\tilde{i}}^\ast L_{i}\right)$ is constant for all $i \neq \tilde{i}$.  Thus, Statement 2 is proven.  Similarly, $\overline{L}_i^\ast \overline{L}_{\tilde{i}} (\overline{L}_{\tilde{i}})^\ast \overline{L}_i$ has spectrum independent of $i \neq \tilde{i}$ (resp. is for all $i \neq \tilde{i}$ equal to a constant multiple times the identity) if and only if $L_i^\ast L_{\tilde{i}}L_{\tilde{i}}^\ast L_{i}$ has spectrum independent of $i \neq \tilde{i}$ (resp. is for all $i \neq \tilde{i}$ equal to a constant multiple times the identity).  Hence, Statements 3 and 4 are proven.\\

If $\{\cW\}_{i=1}^n$ (resp., $\{\overline{\cW}\}_{i=1}^n$) saturates the orthoplex bound, then the maximum value of the chordal distance satisfies for some $i, j \in [n]$
\begin{align*}
\frac{m(k-m)}{k} &= d_c^2(\cW_i,\cW_j) = m - \tr(L_iL_j^\ast L_j L^\ast_i) \\
\bigg(\textrm{resp. } \frac{rm(rk-rm)}{rk} &=  \frac{rm(k-m)}{k}= d_c^2(\overline{\cW}_i,\overline{\cW}_j) = mr - \tr(\overline{L}_i\overline{L}_j^\ast \overline{L}_j \overline{L}^\ast_i) \bigg).
\end{align*}
We can see from Equation~\ref{eqn:crossbloc2} that $\tr(\overline{L}_i\overline{L}_j^\ast \overline{L}_j \overline{L}^\ast_i) = r  \tr(L_iL_j^\ast L_j L^\ast_i)$.  Thus $\{\cW\}_{i=1}^n$ has at least one pair of subspaces at the orthoplex bound if and only if $\{\overline{\cW}\}_{i=1}^n$ does as well.  However, such a configuration is only optimal when $ 2(\cZ(\bF,k)-1) \geq n > \cZ(\bF,k)$  (resp., $2(\cZ(\bF,rk) -1)\geq n > \cZ(\bF,rk)$). Unless $r=1$, $n$ cannot fall in both ranges.

To prove Statement 6, we define 
\begin{align*}
L &= \left( \begin{array}{cccc} L_1 & L_2 & \cdots &L_n \end{array} \right) \in M(\bF,k,mn), \enskip \textrm{and}\\
\overline{L} &= \left( \begin{array}{cccc}\overline{L}_1 & \overline{L}_2 & \cdots & \overline{L}_n \end{array} \right) \in M(\bF,rk,rmn).
\end{align*}
We will make use of the fact that $\{\cW_i\}_{i=1}^n$ (resp., $\{\overline{\cW}_i\}_{i=1}^n$) is a tight fusion frame if and only if $LL^\ast$ (resp., $\overline{L}\overline{L}^\ast$) is a constant multiple of the identity. We calculate the respective matrix products as the sum of the rank one tensors formed from the columns.
\begin{align*}
\overline{L}\overline{L}^\ast &= \sum_{j=1}^m \sum_{i=1}^n \overline{e}_j^i (\overline{e}_j^i)^\ast = \sum_{j=1}^m \sum_{i=1}^n (e_j^i \otimes U_i) (e_j^i \otimes U_i)^\ast \\
&= \sum_{j=1}^m \sum_{i=1}^n \left(\left((e_j^i (e_j^i)^\ast \right)\otimes U_i (U_i)^\ast\right) = \sum_{j=1}^m \sum_{i=1}^n \bigg(\big((e_j^i (e_j^i)^\ast \big)\otimes I_r\bigg) \\
&= \left( \sum_{j=1}^m \sum_{i=1}^n (e_j^i (e_j^i)^\ast \right)\otimes I_r = LL^\ast \otimes I_r.
\end{align*}
It follows that $\overline{L}\overline{L}^\ast = A I_{rk}$ if and only if $LL^\ast = AI_k$.
\end{proof}
We may generalize Theorem~\ref{thm:gentens}.
\begin{cor}\label{cor:gentens} 
Let $\{\cW_i\}_{i=1}^n$ be a collection of $m$-dimensional subspaces in $\bF^k$ and $\{\cV_i\}_{i=1}^n$ be a collection of $r$-dimensional subspaces in $\bF^{\ell}$.  For each $i \in [n]$, fix an orthonormal basis $\{e_j^i\}_{j=1}^m$ of $\cW_i$ and $\{f_t^i\}_{t=1}^r$ of $\cV_i$ .  Further define for each $i \in [n]$ $U_i = (f_1^i \hdots f_r^i)$. Finally for each $i \in [n]$ and $j \in [m]$ define 
\begin{align*}
\overline{e}_j^i &= e_j^i \otimes U_i \in M(\bF,k\ell, r),   \\
\overline{L}_i &= (\overline{e}_1^i \overline{e}_2^i \hdots \overline{e}_m^i) \in M(\bF,k\ell,r m), \textrm{and}\\
\overline{\cW}_i &= \col \overline{L}_i \in Gr(\bF,k\ell,rm).
\end{align*}
Then the following statements hold.
\begin{enumerate}
\item For each $i \in [n]$, the columns of $\overline{L}_i$ are a set of $rm$ orthonormal vectors in $\bF^{rk}$;
\item $\{ \overline{\cW}_i\}_{i=1}^n$ is equichordal (not necessarily a fusion frame) if and only if $\{\cW_i\}_{i=1}^n$ and $\{\cV_i\}_{i=1}^n$  are;
\item $\{ \overline{\cW}_i\}_{i=1}^n$ is strongly simplicial (not necessarily a fusion frame) if and only if $\{\cW_i\}_{i=1}^n$ and $\{\cV_i\}_{i=1}^n$ are; 
\item $\{ \overline{\cW}_i\}_{i=1}^n$ is equiisoclinic (not necessarily a fusion frame) if and only if $\{\cW_i\}_{i=1}^n$ and $\{\cV_i\}_{i=1}^n$ are; 
and
\item[6.] $\{ \overline{\cW}_i\}_{i=1}^n$ is a tight fusion frame if and only if $\{\cW_i\}_{i=1}^n$ and $\{\cV_i\}_{i=1}^n$ are. In this case, the fusion frame bound of $\{ \overline{\cW}_i\}_{i=1}^n$ is the product of the fusion frame bounds of $\{\cW_i\}_{i=1}^n$ and $\{\cV_i\}_{i=1}^n$.
\end{enumerate}
\end{cor}
\begin{proof}
The proof of the corollary follows the proof of Theorem~\ref{thm:gentens} quite closely.  We note that the $U_i$s are partial isometries, so when $i = i'$,  Equation~\ref{eqn:crossbloc} still simplifies, and the off-diagonal blocks of the Gram matrix $\overline{L}_i^\ast \overline{L}_i$ are the zero matrix and the diagonal blocks are the identity.  Thus the columns of $\overline{L}_i$ are orthonormal, and Statement 1 is proven.

For the proof of Statements 2 -- 4, we can only simplify down to Equation~\ref{eqn:crossbloc2b}.  That is,
\[
\overline{L}_i^\ast \overline{L}_{\tilde{i}} (\overline{L}_{\tilde{i}})^\ast \overline{L}_i 
 = \left(L_i^\ast L_{\tilde{i}}L_{\tilde{i}}^\ast L_{i}\right) \otimes \left( U_i^\ast U_{\tilde{i}}U_{\tilde{i}}^\ast U_{i} \right)
\]
We now make use of the fact that for arbitrary square matrices $A$ and $B$, $\tr(A \otimes B) = \tr(A)\tr(B)$ and for arbitrary matrices the singular values of $A \otimes B$ are simply the products of all of the singular values of $A$ with each of the singular values of $B$.

Statement 5 from Theorem~\ref{thm:gentens} does not carry over well.

Finally, we note that Statement 6 is simply [Theorem 4] from \cite{CFMWZ}.
\end{proof}

Another way to generate a fusion frame from another fusion frame is via subspace complementation \cite{BCPST}.
\begin{lem}[\cite{MBI92,QZL05}]\label{lem:eqpa}
Let $\cW_i$ and $\cW_j$ be subspaces of $\bF^k$.  The nonzero principal angles between $\cW_i$ and $\cW_j$ are equal to the nonzero principal angles between their orthogonal complements $\cW^\perp_i$ and $\cW^\perp_j$.
\end{lem}
\begin{prop}\label{prop:orth}
Let $\{\cW_i\}_{i=1}^n \in Gr(\bF,k,m)$.  Then the following statements about $\{\cW^\perp_i\}_{i=1}^n \in Gr(\bF,k,k-m)$ hold.
\begin{enumerate}
\item $\{\cW_i\}_{i=1}^n$ is a tight fusion frame with bound $A$ if and only if $\{\cW^\perp_i\}_{i=1}^n$ is a tight fusion frame with bound $n-A$;
\item $\{\cW_i\}_{i=1}^n$ is equichordal (not necessarily a fusion frame) if and only if $\{\cW^\perp_i\}_{i=1}^n$ is;
\item $\{\cW_i\}_{i=1}^n$ is strongly simplicial (not necessarily a fusion frame) if and only if $\{\cW^\perp_i\}_{i=1}^n$ is; 
\item If $\{\cW_i\}_{i=1}^n$ is equiisoclinic (not necessarily a fusion frame), then $\{\cW^\perp_i\}_{i=1}^n$ is equiisoclinic (not necessarily a fusion frame) if and only if $k=2m$ or the subspaces are trivially all the same; and
\item $\{\cW_i\}_{i=1}^n$ is an orthoplectic Grassmannian packing  if and only if $\{\cW^\perp_i\}_{i=1}^n$ is.
\end{enumerate}
\end{prop}
\begin{proof}
Statement 1 is \cite[Theorem 5]{BCPST}. We note that for tight fusion frames, the fusion frame bound is $\frac{nm}{k}$; hence, $n-A>0$. Statements 2 and 3 follow immediately from Lemma~\ref{lem:eqpa}.  If $\{\cW_i\}_{i=1}^n$ is equiisoclinic, then each pair of subspaces has $m$ equal principal angles.  If the principal angles are all $0$, then the $\cW_i$ are all equal and hence so are the $\cW^\perp_i$.  Otherwise the principal angles are in $(0,\pi/2]$.  By Lemma~\ref{lem:eqpa}, the non-zero principal angles of the $\cW^\perp_i$ are the same, meaning that one must have $\dim \cW^\perp_i = \dim \cW_i$ for equiisoclinicity to be preserved, showing Statement 4.
Finally, for Statement 5, $\{\cW_i\}_{i=1}^n$ is an orthoplectic Grassmannian packing by Theorem~\ref{thm:simplex} if $n > \cZ(\bF,k)$  and there exist $i, j \in [n]$ with $i \neq j$ such that
\[
d_c^2(\cW^\perp_i,\cW^\perp_j)  = d_c^2(\cW_i,\cW_j) = \frac{m(k-m)}{k} = \frac{(k-m)(k-(k-m))}{k}.
\]
Since the number of subspaces and the dimension of the base space $\bF^k$ does not change when taking orthogonal complements, $\{\cW_i\}_{i=1}^n$ being an orthoplectic Grassmannian packing implies that  $\{\cW^\perp_i\}_{i=1}^n$ is as well and vice versa.
\end{proof}
There are two huge differences when comparing the results of subspace complementation with the tensor construction (Theorem~\ref{thm:gentens}).  Initially, in contrast to the tensor construction which always destroys the optimality of an orthoplectic Grassmannian packing, subspace complementation preserves it.  We also note that subspace complementation does not in general preserve equiisoclinicity. We can see that in the following example.
\begin{ex}
Let $\{\cW_i\}_{i=1}^4 \in Gr(\bF,3,1)$ be the subspaces spanned by the columns of the $4 \times 4$ Sylvester-Hadamard matrix with the first row removed:
 \[
 \frac{1}{\sqrt{3}}\left( \begin{array}{cccc} 1 & -1 & 1 & -1 \\1 & 1 & -1 & -1 \\1 & -1 & -1 & 1 \\ \end{array} \right).
\]
This fusion frame of $1$-dimensional subspaces is actually an equiangular tight frame and thus trivially equiisoclinic. Then the columns of the following matrices serve as orthonormal bases for $\{\cW_i^\perp\}_{i=1}^4$ with $a = 1/\sqrt{2}$ and $b=1/\sqrt{6}$
\begin{align*}
L_1 &= \left(\begin{array}{cc} -2b & 0 \\ b & -a \\ b & a \end{array} \right),\enskip L_2= \left(\begin{array}{cc} -2b & 0 \\ -b & -a \\ b & -a \end{array} \right), \\
L_3 &= \left(\begin{array}{cc} -2b & 0 \\ -b & -a \\ -b & a \end{array} \right), \enskip L_4=\left(\begin{array}{cc} -2b & 0 \\ b & -a \\ -b & -a \end{array} \right).
\end{align*}
Although $\{\cW_i^\perp\}_{i=1}^4$ is a strongly simplicial tight fusion frame, it is not equiisoclinic since, for example
\[
L_3^\ast L_2 L_2^\ast L_3 = \left(\begin{array}{cc} \frac{7}{9} & \frac{2}{3\sqrt{3}} \\   \frac{2}{3\sqrt{3}} & \frac{1}{3} \end{array} \right) \neq \alpha I_2.
\]
\end{ex}
On the other hand, Theorem~\ref{prop:orth}.4 is not vacuously true, as seen in the following example.
\begin{ex}
We may use Theorem~\ref{thm:gentens} to generate equiisoclinic systems in any even dimensional space $\bF^{2m}$ consisting of up to $4$ subspaces of dimension $m \in \bN$.  To do this, we start with any equiangular system of $n \geq 2$ vectors in $\bF^2$ and tensor them with $m \times m$ unitaries for any $m \in \bN$.  It follows from Theorem~\ref{thm:gentens}.4 that the result is an equiisoclinic collection of $n$ $m$-dimensional subspaces of $\bF^{2m}$.  Gerzon's bound (Theorem~\ref{thm:simplex}) implies that $n \leq 3$ if $\bF = \bR$ and $n \leq 4$ if $\bF= \bC$; it ends up that equiangular lines exist for each possible $n$.  Any pair of non-orthogonal unit vectors will generate an equiangular system in $\bF^2$ that is not a tight frame, while a pair of orthogonal unit vectors yields an equiangular system that is a tight frame.  The \emph{Mercedes-Benz frame} is an example of an equiangular tight frame of $3$ vectors in $\bR^2 \subset \bC^2$:
\[
L = \left(\begin{array}{ccc}0 & \sqrt{3}/2 & -\sqrt{3}/2 \\ 1 & -1/2 & -1/2 \end{array} \right),
\]
and the \emph{symmetric informationally complete operator-valued measure} in $\bC^2$ is an example of an equiangular tight frame of $4$ vectors in $\bC^2$:
\[
L = \left( \begin{array}{cccc} \frac{\sqrt{3+ \sqrt{3}}}{6} &  e^{\pi i /4} \frac{\sqrt{3- \sqrt{3}}}{6}  &  \frac{\sqrt{3+ \sqrt{3}}}{6} & e^{\pi i /4} \frac{\sqrt{3- \sqrt{3}}}{6} \\
e^{\pi i /4} \frac{\sqrt{3- \sqrt{3}}}{6} & \frac{\sqrt{3+ \sqrt{3}}}{6}  &  -e^{\pi i /4} \frac{\sqrt{3- \sqrt{3}}}{6} & -\frac{\sqrt{3+ \sqrt{3}}}{6}  \end{array} \right).
\]
A particularly nice presentation of an equiisoclinic tight fusion frame of $3$ $4$-dimensional subspaces of $\bC^8$ resulting from (the complement of) a Singer difference set may be found in [Example 3.7] of \cite{FiSh20}.  
\end{ex}

\section*{Acknowledgements}
The author was supported in part by the Explorationsprojekt ``Hilbert Space Frames and Algebraic Geometry'' funded by the Zentrum f\"ur Forschungs\-f\"orderung der Univerist\"at Bremen. The author would like to thank the anonymous referee for their helpful suggestions.


\newcommand{\etalchar}[1]{$^{#1}$}
\providecommand{\bysame}{\leavevmode\hbox to3em{\hrulefill}\thinspace}
\providecommand{\MR}{\relax\ifhmode\unskip\space\fi MR }
\providecommand{\MRhref}[2]{%
  \href{http://www.ams.org/mathscinet-getitem?mr=#1}{#2}
}
\providecommand{\href}[2]{#2}

\end{document}